\documentclass[reqno]{amsart}

\usepackage{amsmath}
\usepackage{amssymb}

\numberwithin{equation}{section}

\newtheorem{Theorem}{Theorem}[section]
\newtheorem{Lemma}[Theorem]{Lemma}
\newtheorem{Proposition}[Theorem]{Proposition}

\theoremstyle{definition}

\newtheorem{Remark}[Theorem]{Remark}

\newcommand{\thref}[1]{Theorem \ref{#1}}
\newcommand{\leref}[1]{Lemma \ref{#1}}
\newcommand{\prref}[1]{Proposition \ref{#1}}
\newcommand{\reref}[1]{Remark \ref{#1}}

\newcommand{\seref}[1]{Section \ref{#1}}

\newcommand{\pd}{{\partial}}

\newcommand{\Rset}{\mathbb{R}}
\newcommand{\Zset}{\mathbb{Z}}
\newcommand{\Nset}{\mathbb{N}}

\newcommand{\Id}{\mathrm{Id}}

\newcommand{\Span}{\mathrm{span}}
\newcommand{\Wr}{\mathrm{Wr}}

\newcommand{\al}{{\alpha}}
\newcommand{\be}{{\beta}}
\newcommand{\la}{{\lambda}}

\newcommand{\de}{{\delta}}
\newcommand{\De}{{\Delta}}
\newcommand{\na}{\nabla}
\newcommand{\taub}{\bar{\tau}}
\newcommand{\ep}{{\epsilon}}
\newcommand{\var}{\varepsilon}

\newcommand{\dms}{d\mu_{\mathrm{d}}(s)}

\newcommand{\fh}{\hat{f}}
\newcommand{\qh}{\hat{q}}

\newcommand{\pab}{p^{\al,\be}}
\newcommand{\qaba}{q^{\al,\be;a}}
\newcommand{\cM}{{\mathcal M}}
\newcommand{\cD}{{\mathcal D}}
\newcommand{\cPd}{{\mathcal P}^{d}}
\newcommand{\cH}{{\mathcal H}}
\newcommand{\cKd}{{\mathcal K}^{d}}
\newcommand{\cAd}{{\mathcal A}^{d}}
\newcommand{\cDaba}{{\mathcal D}^{\alpha,\beta;a}}
\newcommand{\cDh}{\hat{{\mathcal D}}}
\newcommand{\cA}{{\mathcal A}}
\newcommand{\cB}{{\mathcal B}}
\newcommand{\cBb}{\bar{{\mathcal B}}}
\newcommand{\fD}{{\mathfrak D}}
\newcommand{\fDb}{{\mathfrak D}_{\beta}}
\newcommand{\fDah}{{\hat{\mathfrak D}}_{\alpha}}
\newcommand{\cAaba}{{\mathcal A}^{\alpha,\beta;a}}
\newcommand{\rb}{\bar{r}}

\newcommand{\Bab}{B_{\al,\be}(z,\pd_z)}
\newcommand{\Babs}{B_{\al,\be,s}(z,\pd_z)}
\newcommand{\Lab}{L_{\al,\be}(n,E_n)}
\newcommand{\laab}{\la^{\al+\be}}

\newcommand{\cL}{{\mathcal L}}
\newcommand{\cQ}{{\mathcal Q}}
\newcommand{\cP}{{\mathcal P}}
\newcommand{\cLh}{{\hat{\mathcal L}}}
\newcommand{\Lh}{{\hat{L}}}

\begin{document}
\title[Krall-Jacobi algebras of partial differential operators]{Krall-Jacobi commutative algebras of partial differential operators}

\date{March 21, 2011}

\author[P.~Iliev]{Plamen~Iliev}
\address{School of Mathematics, Georgia Institute of Technology,
Atlanta, GA 30332--0160, USA}
\email{iliev@math.gatech.edu}
\thanks{The author was supported in part by NSF Grant \#0901092 and Max-Planck-Institut f\"ur Mathematik, Bonn.}

\begin{abstract}
We construct a large family of commutative algebras of partial differential operators invariant under rotations. These algebras are isomorphic extensions of the algebras of ordinary differential operators introduced by Gr\"unbaum and Yakimov corresponding to Darboux transformations at one end of the spectrum of the recurrence operator for the Jacobi polynomials. The construction is based on a new proof of their results which leads to a more detailed description of the one-dimensional theory. In particular, our approach establishes a conjecture by Haine concerning the explicit characterization of the Krall-Jacobi algebras of ordinary differential operators. 
\end{abstract}

\maketitle

\tableofcontents

\section{Introduction}\label{se1}

It is practically impossible to list the numerous applications of Jacobi polynomials which were introduced more than 150 years ago as solutions of the hypergeometric equation \cite{J}. In 1938, Krall \cite{Kr1} studied the general problem of classifying orthogonal polynomials which are eigenfunctions of a higher-order differential operator and solved it completely for operators of order 4, thus extending the classical orthogonal polynomials \cite{Kr2}. In the last century, many different solutions of Krall's problem were constructed, see for instance \cite{KK,Ko,L,Z} and the references therein.

More recently, new techniques have emerged in the literature inspired by the bispectral problem \cite{DG}.  The most general result concerning the Krall's problem appeared not long ago in the beautiful work of Gr\"unbaum and Yakimov \cite{GY}. They constructed a large family of solutions to a differential-difference bispectral problem, containing as special or limiting cases all previously known solutions of Krall's problem. Their approach is based on a very subtle application of the Darboux transformation \cite{D}, one of the basic tools in the theory of solitons \cite{MS}. The construction naturally leads to a commutative algebra of ordinary differential operators diagonalized by the generalized Jacobi polynomials, which we call the Krall-Jacobi algebra.

In the present paper we discuss an extension of the above theory within the context of partial differential operators invariant under rotations. First, we give a new proof of the results in \cite{GY}, which shows that if we iterate the Darboux transformation at one end of the spectrum of the Jacobi recurrence operator, then the Krall-Jacobi commutative algebra of differential operators is contained in an associative algebra with two generators which have natural multivariate extensions. This fact does not seem to follow easily even from the explicit formulas in \cite{KK} for the operator of minimal order in the Krall-Jacobi algebra in the case of a single Darboux transformation. Our proof is based on the approach used in \cite{I1,I2} to establish the bispectrality for rank-one commutative algebras of difference or  $q$-difference operators. The main difficulty here is to evaluate certain discrete integrals (or sums) involving the Jacobi polynomials, which were trivial integrals involving exponents in the rank-one case. As another corollary of our construction we establish Conjecture~3.2 on page 161 in \cite{H} for the Krall-Jacobi algebra of ordinary differential operators, which gives an explicit characterization of the isomorphic (dual) algebra of eigenfunctions. Moreover, our techniques yield an explicit eigenbasis of multivariate polynomials for the Krall-Jacobi algebras of partial differential operators in terms of the quantities used to describe the sequence of one-dimensional Darboux transformations and the spherical harmonics. 

The paper is organized as follows. In \seref{se2} we introduce the Jacobi polynomials and the corresponding recurrence and differential operators. In \seref{se3} we review briefly the sequence of Darboux transformations from the Jacobi recurrence operator. In \seref{se4} we present the new proof of the results in \cite{GY} together with the additional properties of the Krall-Jacobi algebra mentioned above. We give a detailed proof for the case needed for the multivariate extension (i.e. Darboux transformations at one end of the spectrum), but we indicate the necessary modifications for Darboux transformations at both ends in \reref{re4.6}. In \seref{se5} we define the Krall-Jacobi algebras of partial differential operators and we write an explicit basis in the space of polynomials in several variables. In the last section we illustrate the constructions in the paper with the simplest possible example which leads to a multivariate analog of Krall polynomials \cite{Kr2}.

\section{Jacobi polynomials and operators} \label{se2}

Throughout the paper we denote by $\pab_n(z)$ the Jacobi polynomials 
normalized as follows
\begin{equation}\label{2.1}
\pab_n(z)=(-1)^n\frac{(\al+\be+1)_n}{n!}F(-n,n+\al+\be+1,\be+1;z),
\end{equation}
where $F$ is the Gauss' ${}_2F_1$ hypergeometric function and 
$(a)_n$ is the the shifted factorial $(a)_0=1$, $(a)_n=a(a+1)\cdots(a+n-1)$ 
for $n>0$. Besides the orthogonality on $[0,1]$ with respect to the 
measure $(1-z)^{\al}z^{\be}dz$ the Jacobi polynomials are eigenfunctions 
of the second-order differential operator 
\begin{equation}\label{2.2}
\Bab=z(z-1)\pd_z^2+(z(\al+\be+2)-(\be+1))\pd_z
\end{equation}
with eigenvalue
\begin{equation}\label{2.3}
\laab_n=n(n+\al+\be+1),
\end{equation}
i.e. we have
\begin{equation}\label{2.4}
\Bab\pab_n(z)=\laab_n\pab_n(z).
\end{equation}
In order to simplify the notation, we shall write $\la_n$ instead of $\laab_n$ unless we need to use different values for the parameters and the explicit dependence on $\al+\be$ is important.

Since any family of orthogonal polynomials satisfies a three-term recurrence 
relation, the polynomials $\pab_n(z)$ are also eigenfunctions for a 
second-order difference operator acting on the discrete variable $n$. 
More precisely, if we denote by $E_n$ the customary shift operator acting on 
functions of a discrete variable $n$ by 
$$E_nf_n=f_{n+1},$$
then 
\begin{equation}\label{2.5}
\Lab\pab_n(z)=z\pab_n(z),
\end{equation}
where $\Lab$ is the second-order difference operator
\begin{equation}\label{2.6}
\Lab =A_nE_n+B_n\Id+C_nE_n^{-1},
\end{equation}
with coefficients
\begin{subequations}\label{2.7}
\begin{align}
&A_n=\frac{(n+1)(n+\be+1)}{(2n+\al+\be+1)(2n+\al+\be+2)}\label{2.7a}\\
&C_n=\frac{(n+\al)(n+\al+\be)}{(2n+\al+\be)(2n+\al+\be+1)}\label{2.7b}\\
&B_n=1-A_n-C_n.\label{2.7c}
\end{align}
\end{subequations}
Here and later we assume that $\pab_n(z)=0$ for $n<0$. Equivalently, if we 
think of $\pab_n(z)$ as the semi-infinite vector 
$$[\pab_0(z),\pab_1(z),\pab_2(z),\dots]^t,$$
then $\Lab$ can be represented by the tridiagonal semi-infinite (Jacobi) 
matrix
\begin{equation}\label{2.8}
\Lab=\left[\begin{matrix}
B_0&A_0\\
C_1&B_1&A_1\\
   &C_2&B_2&A_2\\
   &   &\ddots&\ddots&\ddots
\end{matrix}\right].
\end{equation} 
In view of equations \eqref{2.4} and \eqref{2.5} we can say that 
the polynomials $\pab_n(z)$ solve a {\em differential-difference 
bispectral problem}.

Using the explicit formula \eqref{2.1} one can easily check that the Jacobi 
polynomials satisfy the following differential-difference equation
\begin{equation}\label{2.9}
[2z\pd_z+(\al+\be)](\pab_{n}(z)-\pab_{n-1}(z))=(2n+\al+\be)(\pab_n(z)+\pab_{n-1}(z)).
\end{equation}
The above formula will be a key ingredient in the computation of 
``discrete integrals'' involving the Jacobi polynomials, which are analogous
to standard (continuous) integrals of products of polynomials and exponents. This is the 
first place where the particular normalization of $\pab_n(z)$ is very 
important.

\section{Discrete Darboux transformations}\label{se3}

In this section, following \cite{GY}, we describe the result of $k$ successive Darboux transformations starting
from $\Lab$ at $z=1$ for $\al\in\Nset$ and $k\leq\al$.

Recall that the lattice version of the Darboux transformation \cite{MS} of a difference operator (bi-infinite matrix) $\cL_0$ at $z$ amounts to performing an upper-lower factorization of $\cL_0-z\Id$ and to producing a new matrix by exchanging the factors. If we iterate this process $k$ times starting
from $\cL_0$ at $z=1$, we obtain a difference operator $\cLh$ as follows:
\begin{align}
&\cL_0=\Id+\cP_0\cQ_0\curvearrowright \cL_1=\Id+\cQ_0\cP_0=\Id+\cP_1\cQ_1
\curvearrowright\cdots                      \nonumber\\
&\quad \cL_{k-1}=\Id+\cQ_{k-2}\cP_{k-2}=\Id+\cP_{k-1}\cQ_{k-1}
                                \label{3.1}\\
&\qquad\curvearrowright \cLh=\cL_k=\Id+\cQ_{k-1}\cP_{k-1}.  \nonumber
\end{align}
From \eqref{3.1} it follows that 
\begin{equation}\label{3.2}
\cLh \cQ=\cQ\cL_0
\end{equation}
and 
\begin{equation}\label{3.3}
(\cL_0-\Id)^k=\cP\cQ, 
\end{equation}
where $\cP=\cP_0\cP_1\cdots \cP_{k-1}$ and $\cQ=\cQ_{k-1}\cQ_{k-2}\cdots \cQ_{0}$. The above 
formulas imply that 
\begin{equation}\label{3.4}
\ker(\cQ)\subset\ker((\cL_0-\Id)^k)\text{ and }\cL_0(\ker(\cQ))\subset\ker(\cQ).
\end{equation}
Conversely, one can show that if \eqref{3.4} holds then there exists a 
difference operator (bi-infinite matrix) $\cLh$ such that \eqref{3.2} holds and this operator 
is obtained by a sequence of Darboux transformations as in \eqref{3.1}.

Note that, up to a factor, the lower-triangular bi-infinite matrix $\cQ$ is uniquely 
determined by its kernel, i.e. if $\{\psi^{(j)}_n\}_{j=0}^{k-1}$ is a basis for 
$\ker(\cQ)$ then 
\begin{equation*}
\cQ f_n= g_n\Wr_n(\psi^{(0)}_n,\psi^{(1)}_n,\dots,\psi^{(k-1)}_n,f_n),
\end{equation*}
for an appropriate function (diagonal bi-infinite matrix) $g_n$. We use $\Wr_n$ to denote the discrete Wronskian (or Casorati determinant):
$$\Wr_n(h^{(1)}_n,h^{(2)}_n,\dots,h^{(k)}_n)=\det(h^{(i)}_{n-j+1})_{1\leq i,j\leq k}.$$
Combining the above remarks, we see that the sequence of Darboux 
transformations \eqref{3.1} for $\cL_0$ is characterized by choosing a 
basis for $\ker(\cQ)$ satisfying
\begin{equation}\label{3.5}
(\cL_0-\Id)\psi^{(0)}_n=0\text{ and }(\cL_0-\Id)\psi^{(j)}_n=\psi^{(j-1)}_n
\text{ for }j=1,\dots,k-1.
\end{equation}

We apply the above construction for the difference operator $\cL_0=\cL_{\al,\be}(n+\var,E_n)=A_{n+\var}E_n+B_{n+\var}\Id+C_{n+\var}E_n^{-1}$. Note that for a generic $\var$, the coefficients $A_{n+\var}$, $B_{n+\var}$ and $C_{n+\var}$ are defined for all $n\in\Zset$ and therefore $\cL_{\al,\be}(n+\var,E_n)$ is a well defined second-order difference operator (or a bi-infinite tridiagonal matrix).

For  $j\in\{0,1,\dots,k-1\}$ we set
\begin{subequations}\label{3.6}
\begin{align}
\phi^{1,j}_n
&=\frac{(-1)^j(n+1)_j(-n-\al-\be)_j}{j!(1-\al)_j},\label{3.6a}\\
\phi^{2,j}_n
&=\frac{(-1)^j(n+1)_{\al}(n+\be+1)_{\al}(-n)_j(n+\al+\be+1)_j}
{j!\al!(1+\al)_j(1+\be)_{\al}}.\label{3.6b}
\end{align}
\end{subequations}
One can show that the functions $\phi^{i,j}_{n+\var}$ are linearly independent 
(as functions of $n$) and 
\begin{align*}
&(\cL_{\al,\be}(n+\var,E_n)-\Id)\phi^{i,0}_{n+\var}=0,\text{ for }i=1,2\\
&(\cL_{\al,\be}(n+\var,E_n)-\Id)\phi^{i,j}_{n+\var}=\phi^{i,j-1}_{n+\var},
\text{ for }j=1,\dots,k-1, \quad i=1,2.
\end{align*}
Thus, we can write the functions $\psi^{(j)}_n$ as a linear combination of 
$\phi^{i,j}_{n+\var}$ as follows
\begin{equation*}
\psi^{(j)}_n=\sum_{l=0}^{j}(a_{j-l}\phi^{1,l}_{n+\var}+b_{j-l}\phi^{2,l}_{n+\var}).
\end{equation*}
If $b_0=0$ then $\psi^{(0)}_n=a_0\neq 0$, i.e. we can take 
$\psi^{(0)}_n=1$ and one can check that $\cL_1$ in \eqref{3.1} coincides 
(up to a conjugation by a diagonal matrix) with $\cL_{\al-1,\be}(n+\var,E_n)$. Therefore the operator 
$\cLh=\cL_k$ can be obtained by a sequence of $k-1$ Darboux transformations 
starting from $\cL_{\al-1,\be}(n+\var,E_n)$. Thus we can assume that 
$b_0\neq 0$, hence we can take $b_0=1$. Since $\cQ$ depends only on the space
$\Span\{\psi^{(0)}_n,\psi^{(1)}_n,\dots, \psi^{(k-1)}_n\}$, and not on the choice
of the specific basis, we can can take $b_j=0$ for $j>0$. Thus we shall 
consider a basis for $\ker(\cQ)$ of the form
\begin{equation*}
\psi^{(j)}_n=\sum_{l=0}^{j}a_{j-l}\phi^{1,l}_{n+\var}+\phi^{2,j}_{n+\var},
\end{equation*}
depending on $\al\in\Nset$, $\be$ and $k$ free parameters $a_0,a_1,\dots,a_{k-1}$. 

Now we consider the limit $\var\rightarrow 0$. Since $\lim_{\var\rightarrow 0}A_{-1+\var}=0$, it follows that the intertwining relation \eqref{3.2} holds 
for the semi-infinite parts of the bi-infinite matrices $\cL_{0}$, $\cLh$ and $\cQ$. In other words, we have
\begin{equation}\label{3.7}
\Lh(n,E_n) Q=Q \Lab,
\end{equation}
where  $\Lab$  is the semi-infinite matrix in \eqref{2.8}, $\Lh(n,E_n)$ is a similar semi-infinite Jacobi matrix, and $Q$ is a lower-triangular semi-infinite matrix, acting on functions (vectors) $f_n=[f_0,f_1,\dots]^t$ by 
\begin{equation}\label{3.8}
Q f_n= g_n\Wr_n(\psi^{(0)}_n,\psi^{(1)}_n,\dots,\psi^{(k-1)}_n,f_n),
\end{equation} 
where
\begin{equation}\label{3.9}
\psi^{(j)}_n=\sum_{l=0}^{j}a_{j-l}\phi^{1,l}_{n}+\phi^{2,j}_{n},
\end{equation}
with the convention that $f_j=0$ when $j<0$.

We shall normalize the matrix $Q$ by taking $g_n=1$ in \eqref{3.8}. Finally, we 
denote by $q_n(z)$ the polynomials defined by 
\begin{equation}\label{3.10}
q_n(z)=Q \pab_n(z)=
\Wr_n(\psi^{(0)}_n,\psi^{(1)}_n,\dots,\psi^{(k-1)}_n,\pab_n(z)),
\end{equation}
which depend on the free parameters $\al\in\Nset$, $\be$ and 
$a=(a_0,a_1,\dots,a_{k-1})$. 
\begin{Remark}\label{re3.1}
To simplify the notation we omit the explicit dependence of the parameters $\al$, $\be$ and $a$ in the functions $\psi^{(j)}_n$ and $q_n(z)$. When these parameters are needed, we shall write 
$\psi^{(j);\al,\be;a}_n(z)$ and $\qaba_n(z)$.
\end{Remark}
\begin{Remark}\label{re3.2}
From \eqref{2.5} and \eqref{3.7} it follows 
that 
\begin{equation}\label{3.11}
\Lh(n,E_n)q_n(z)=zq_n(z).
\end{equation}
It is easy to see that the off-diagonal entries of the tridiagonal matrix $\Lh(n,E_n)$ are 
nonzero and therefore, by Favard's theorem, there exists a unique (up to a
multiplicative constant) moment functional $\cM$ for which 
$\{q_n(z)\}_{n=0}^{\infty}$ is an orthogonal sequence, i.e.
\begin{equation*}
\cM(q_nq_m)=0, \text{ for }n\neq m \text{ and }
\cM(q_n^2) \neq 0.
\end{equation*}
More precisely, one can show that there exist constants 
$u_0,u_1,\dots,u_{k-1}$ 
such that the moment functional $\cM$ is given by the weight distribution 
\begin{equation*}
w(z)=
(1-z)^{\al-k}z^{\be}
+\sum_{j=0}^{k-1}u_j\de^{(j)}(z-1),
\end{equation*}
where $\de$ is the Dirac delta function. The parameters $u_j$ correspond to a 
different parametrization of the Darboux transformation \eqref{3.1}. The 
proof of this statement is well known and we omit the details. We refer 
the reader to \cite[Theorem 2, p.~287]{GHH} where a similar result was proved 
for extensions of Laguerre polynomials. 
\end{Remark}

\section{The commutative algebras $\cAaba$ and $\cDaba$}\label{se4}

In this section we prove that the polynomials $q_n(z)$ are eigenfunctions 
for all operators in a commutative algebra of differential operators.

\subsection{Statement of the main one-dimensional theorem}\label{ss4.1}
In order to formulate the precise statement let us denote
\begin{equation}\label{4.1}
\cD_{1}=z\pd_z \text{ and }\cD_{2}=z\pd_z^2+(\be+1)\pd_z,
\end{equation}
and let $\fDb$ denote the associative algebra generated by 
$\cD_{1}$ and $\cD_{2}$, i.e.
\begin{equation}\label{4.2}
\fDb=\Rset\langle \cD_{1},\cD_{2}\rangle.
\end{equation}
It is easy to see that 
\begin{equation}\label{4.3}
[\cD_{2},\cD_{1}]=\cD_{2}.
\end{equation}
Next we set
\begin{equation}\label{4.4}
\tau_n=\Wr_n(\psi^{(0)}_n,\psi^{(1)}_n,\dots,\psi^{(k-1)}_n)
\end{equation}
where the functions $\psi^{(j)}_n$ are given in \eqref{3.9}. Here we use the same convention as in \reref{re3.1} and we omit the explicit dependence of the parameters $\al$, $\be$ and $a$. Note 
that $\tau_n$ defined above and $\la_n$ defined in \eqref{2.3} are polynomials 
of $n$. In fact, one can show that up to a simple factor, $\tau_n$ 
belongs to $\Rset[\laab_{n-(k-1)/2}]=\Rset[\la_{n}^{\al+\be-k+1}]$, where 
$\la_n^{s}=n(n+s+1)$. One simple way to see this is to use the involution introduced 
in \cite{GY} which characterizes the 
subring $\Rset[\la_{n}^{s}]$ in $\Rset[n]$. For $s\in\Rset$ we 
define $I^{(s)}$ on $\Rset[n]$ by 
$$I^{(s)}(n)=-(n+s+1).$$
Then clearly $I^{(s)}(\la_n^s)=\la_n^s$ hence every polynomial of $\la_n^s$ is 
invariant under the action of $I^{(s)}$. Conversely, if $r\in\Rset[n]$
is invariant under $I^{(s)}$, then $r\in\Rset[\la_{n}^s]$.

Note also that the functions $\phi^{i,j}_n$ are invariant under 
$I^{(\al+\be)}$, 
and therefore $\psi^{(j)}_n$ in \eqref{3.9} are also invariant under 
$I^{(\al+\be)}$. From this it follows that $I^{(\al+\be-k+1)}$ will 
reverse the order of the rows in the determinant in equation \eqref{4.4}, leading to
$$I^{(\al+\be-k+1)}(\tau_{n})=(-1)^{k(k-1)/2}\tau_n.$$
The last formula shows that if $k\equiv 0,1\mod 4$, then 
$\tau_n\in\Rset[\la_{n-(k-1)/2}]$. Otherwise, $\tau_n$ is divisible by 
$(2n+\al+\be-k+2)=\la_{n-k/2+1}-\la_{n-k/2}$ and the quotient belongs 
to $\Rset[\la_{n-(k-1)/2}]$. Summarizing these observations we see that 
\begin{equation*}
\tau_n=\ep^{(k)}_n \taub(\la_{n-(k-1)/2}), 
\end{equation*}
where $\taub$ is a polynomial and
\begin{equation}\label{4.5}
\ep^{(k)}_n=\begin{cases}
1& \text{ if }k\equiv 0,1\mod 4\\
\la_{n-k/2+1}-\la_{n-k/2}& \text{ if }k\equiv 2,3\mod 4.
\end{cases}
\end{equation}

For $n,m\in\Zset$ and for a function $f_s$ defined on $\Zset$ it will 
be convenient to use the following notation
\begin{equation*}
\int_{m}^{n}f_s\dms=\begin{cases}
\sum_{s=m+1}^{n}f_s &\text{if }n>m\\
0 &\text{if }n=m\\
-\sum_{s=n+1}^{m}f_s&\text{if }n<m.
\end{cases}
\end{equation*}
Thus 
\begin{equation}\label{4.6}
\int_m^{n}f_s\dms=f_n+\int_m^{n-1}f_s\dms\text{ for all }m, n\in\Zset.
\end{equation}

We denote by $\cAaba$ the algebra of all polynomials $f$ such that 
$f(\la_{n-k/2})-f(\la_{n-k/2-1})$ is divisible by $\tau_{n-1}$ in 
$\Rset[n]$:
\begin{equation}\label{4.7}
\cAaba=\left\{f\in\Rset[t]:\frac{f(\la_{n-k/2})-f(\la_{n-k/2-1})}{\tau_{n-1}} \in\Rset[n]\right\}.
\end{equation}

\begin{Remark}\label{re4.1}
It is not hard to see that $\cAaba$ contains a polynomial of every degree 
greater than $\deg(\taub)$. Indeed, notice that for every $r\in\Rset[t]$ 
there exists $\bar{r}\in\Rset[t]$ such that
\begin{equation*}
\frac{r(\la_n)-r(\la_{n-1})}{\la_n-\la_{n-1}}=\bar{r}(\la_{n-1/2}).
\end{equation*}
Conversely, for every polynomial $\bar{r}$ there exists a polynomial $r$ 
such that the above equation holds. Moreover, up to an additive constant, 
$r$ is uniquely determined by 
$$r(\la_n)=\int_{0}^{n}(\la_s-\la_{s-1})\bar{r}(\la_{s-1/2})\dms
=\int_{-1}^{n-1}(\la_{s+1}-\la_{s})\bar{r}(\la_{s+1/2})\dms.$$ 
If 
$\bar{r}\neq 0$  then $\deg(r)=\deg(\bar{r})+1$. 
Thus, for every $g\in\Rset[t]$ and $c\in\Rset$ we can define $f\in\cAaba$ by
\begin{equation*}
f(\la_{n-k/2})=\int_{0}^{n}\ep^{(k+2)}_{s}g(\la_{s-(k+1)/2})
\tau_{s-1}\dms+c.
\end{equation*}
Conversely, to every $f\in\cAaba$ there correspond unique $g\in\Rset[t]$ and $c\in\Rset$, 
so that the above equation holds. 
\end{Remark}

The main result in this section is the following theorem.
\begin{Theorem}\label{th4.2}
For every $f\in\cAaba$ there exists $B_f=B_f(\cD_{1},\cD_{2})\in\fDb$ such that
\begin{equation}\label{4.8}
B_fq_n(z)=f(\la_{n-k/2})q_n(z).
\end{equation}
Thus, $\cDaba=\{B_f:f\in \cAaba\}$ is a commutative subalgebra of $\fDb$, isomorphic 
to $\cAaba$.
\end{Theorem}
Equations \eqref{3.11} and \eqref{4.8} establish the bispectral properties of the generalized Jacobi polynomials $q_n(z)$. 
The fact that for every $f\in\cAaba$ there exists a differential operator $B_f$ satisfying \eqref{4.8} was conjectured in \cite{H}. The fact that $B_f\in\fDb$ is crucial for the multivariate extensions.

\subsection{Auxiliary facts}\label{ss4.2}
For the proof of the above theorem we shall need several lemmas. 
First we formulate a discrete analog of a lemma due to Reach \cite{R}. 
\begin{Lemma}\label{le4.3}
Let $f^{(0)}_{n}, f^{(1)}_{n},\dots,f^{(k+1)}_{n}$ be functions of a discrete 
variable $n$. Fix $n_1,n_2,\dots,n_{k+1}\in\Zset$ and let 
\begin{equation}\label{4.9}
F_n=\sum_{j=1}^{k+1}(-1)^{k+1+j}f^{(j)}_n\int_{n_j}^{n}f^{(0)}_{s}
\Wr_s(f^{(1)}_{s},\dots,\fh^{(j)}_{s},
\dots,f^{(k+1)}_{s})\dms,
\end{equation}
with the usual convention that the terms with hats are omitted. 
Then 
\begin{equation}\label{4.10}
\begin{split}
\Wr_n(f^{(1)}_{n},\dots,f^{(k)}_{n},F_n)=&
\int_{n_{k+1}}^{n-1}f^{(0)}_{s}\Wr_s(f^{(1)}_{s},\dots,f^{(k)}_{s})\dms\\
&\times \Wr_n(f^{(1)}_{n},\dots,f^{(k+1)}_{n}).
\end{split}
\end{equation}
\end{Lemma}
The above lemma was used in \cite{I1} to give an alternative proof of a theorem in \cite{HI} that rank-one commutative rings of difference operators with unicursal spectral curves are bispectral. Similar argument was used also in \cite{I2} for rank-one commutative rings of $q$-difference operators. Since the application 
in our case is very subtle and we need different elements of the proof, we 
briefly sketch it below. 
\begin{proof}[Proof of \leref{le4.3}]
Note that
\begin{equation*}
\left|\begin{matrix}
f^{(1)}_{n}&f^{(2)}_{n}&\dots &f^{(k+1)}_{n}\\
f^{(1)}_{n-1}&f^{(2)}_{n-1}&\dots &f^{(k+1)}_{n-1}\\
\vdots & \vdots& &\vdots\\
f^{(1)}_{n-k+1}&f^{(2)}_{n-k+1}&\dots &f^{(k+1)}_{n-k+1}\\
f^{(1)}_{n-l}&f^{(2)}_{n-l}&\dots &f^{(k+1)}_{n-l}
\end{matrix}\right|=0, \text{ for every }l=0,1,\dots,k-1.
\end{equation*}
Expanding the above determinant along the last row we obtain
\begin{equation}\label{4.11}
\sum_{j=1}^{k+1}(-1)^{k+1+j}f^{(j)}_{n-l}
\Wr_n(f^{(1)}_{n},\dots\fh^{(j)}_{n},\dots,f^{(k+1)}_{n})=0
\text{ for }l=0,\dots,k-1.
\end{equation}
Using \eqref{4.6} and \eqref{4.11} we see that for $l=0,\dots,k-1$ we have
\begin{equation}\label{4.12}
F_{n-l}=\sum_{j=1}^{k+1}(-1)^{k+1+j}f^{(j)}_{n-l}\int_{n_j}^{n}f^{(0)}_{s}
\Wr_s(f^{(1)}_{s},\dots,\fh^{(j)}_{s},
\dots,f^{(k+1)}_{s})\dms,
\end{equation}
and 
\begin{equation}\label{4.13}
\begin{split}
F_{n-k}=&\sum_{j=1}^{k+1}(-1)^{k+1+j}f^{(j)}_{n-k}\int_{n_j}^{n}f^{(0)}_{s}
\Wr_s(f^{(1)}_{s},\dots,\fh^{(j)}_{s},
\dots,f^{(k+1)}_{s})\dms \\
&\qquad-f^{(0)}_{n}\Wr_n(f^{(1)}_{n},\dots,f^{(k+1)}_{n}).
\end{split}
\end{equation}
If we plug \eqref{4.12} and \eqref{4.13} in 
$\Wr_n(f^{(1)}_{n},\dots,f^{(k)}_{n},F_n)$, then most of the terms cancel by 
column elimination and we obtain \eqref{4.10}.
\end{proof}

\begin{Remark}\label{re4.4}
We list below important corollaries from the proof of \leref{le4.3}.

\item[(i)] Note that the right-hand side of \eqref{4.10} does not depend 
on the integers $n_1,n_2\dots,n_k$. Moreover, if we change $n_{k+1}$ then 
only the value of 
$$\int_{n_{k+1}}^{n-1}f^{(0)}_{s}\Wr_s(f^{(1)}_{s},\dots,f^{(k)}_{s})\dms$$
will change by an additive constant, which is independent of $n$ and 
$f^{(k+1)}_{n}$. Thus, instead of \eqref{4.9} we can write 
\begin{equation*}
F_n=\sum_{j=1}^{k+1}(-1)^{k+1+j}f^{(j)}_n\int^{n}f^{(0)}_{s}
\Wr_s(f^{(1)}_{s},\dots,\fh^{(j)}_{s},
\dots,f^{(k+1)}_{s})\dms,
\end{equation*}
leaving the lower bounds of the integrals (sums) blank and we can fix them 
at the end appropriately. This would allow us to easily change the variable, 
without keeping track of the lower end.
\item[(ii)] From \eqref{4.11} it follows that for every $l=-1, 0,1\dots,k-1$ we 
can write $F_n$ also as
\begin{equation*}
F_n=\sum_{j=1}^{k+1}(-1)^{k+1+j}f^{(j)}_n\int^{n+l}f^{(0)}_{s}
\Wr_s(f^{(1)}_{s},\dots,\fh^{(j)}_{s},
\dots,f^{(k+1)}_{s})\dms,
\end{equation*}
and changing the variable in the discrete integral we obtain
\begin{equation*}
F_n=\sum_{j=1}^{k+1}(-1)^{k+1+j}f^{(j)}_n\int^{n}f^{(0)}_{s+l}
\Wr_s(f^{(1)}_{s+l},\dots,\fh^{(j)}_{s+l},
\dots,f^{(k+1)}_{s+l})\dms.
\end{equation*}
\item[(iii)] Suppose now that $k$ is even.
Using (ii) with $l=k/2-1$ we have
\begin{equation}\label{4.14}
F_n=\sum_{j=1}^{k+1}(-1)^{k+1+j}f^{(j)}_n\int^{n}f^{(0)}_{s+k/2-1}
\Wr_s(f^{(1)}_{s+k/2-1},\dots,\fh^{(j)}_{s+k/2-1},
\dots,f^{(k+1)}_{s+k/2-1})\dms.
\end{equation}
Let us consider the sum consisting of the first $k$ integrals:
\begin{equation*}
F_n^{(k)}=\sum_{j=1}^{k}(-1)^{k+1+j}f^{(j)}_n\int^{n}f^{(0)}_{s+k/2-1}
\Wr_s(f^{(1)}_{s+k/2-1},\dots,\fh^{(j)}_{s+k/2-1},
\dots,f^{(k+1)}_{s+k/2-1})\dms.
\end{equation*}
Expanding each Wronskian determinant along the last column we can write 
$F_n^{(k)}$ as a sum of $k$ terms $F_n^{(k,m)}$, each one involving as 
integrand one of the functions $f^{(k+1)}_{s+m}$, where 
$m=-k/2,-k/2+1,\dots,k/2-1$.
We can use \eqref{4.11} once again, this time for the functions 
$f^{(0)}_{n}, f^{(1)}_{n},\dots,f^{(k)}_{n}$ (i.e. omitting $f^{(k+1)}_{n}$),
to change $s$ as follows:\\
\begin{itemize}
\item If $m\geq 0$ we can replace $n$ with $n-m$ in the upper limit of the 
integral, or equivalently, if we keep the upper limit of the integral to be 
$n$, we can replace $s$ by $s-m$ in the integrand. Thus $F_n^{(k,m)}$ will 
have $f^{(k+1)}_{s}$ 
as integrand (and the integration goes up to $n$).
\item If $m\leq -1$ we can replace $s$ by $s-m-1$, thus $F_n^{(k,m)}$ will have 
$f^{(k+1)}_{s-1}$ as integrand (and the integration goes up to $n$).
\end{itemize}
This means that we can rewrite $F^{(k)}_{n}$ as sums of integrals, and the 
integrands can be combined in pairs (corresponding to $m$ and $(-m-1)$ 
for $m=0,1,\dots,k/2$) involving $f^{(k+1)}_{s}$ and 
$f^{(k+1)}_{s-1}$. Explicitly, we can write $F_n^{(k)}$, as a sum of terms 
which, up to a sign, have the form
\begin{equation}\label{4.15}
\begin{split}
&f^{(1)}_n\int^{n}\Bigg[f^{(0)}_{s-m+k/2-1}\left|\begin{matrix}
f^{(2)}_{s+k/2-m-1}&\dots &f^{(k)}_{s+k/2-m-1}\\
f^{(2)}_{s+k/2-m-2}&\dots &f^{(k)}_{s+k/2-m-2}\\
\vdots& &\vdots\\
\fh^{(2)}_{s}&\dots &\fh^{(k)}_{s}\\
\vdots& &\vdots\\
f^{(2)}_{s-k/2-m}&\dots &f^{(k)}_{s-k/2-m}
\end{matrix}\right|f^{(k+1)}_{s}\\
&\qquad\qquad\qquad-f^{(0)}_{s+m+k/2-1}\left|\begin{matrix}
f^{(2)}_{s+k/2+m-1}&\dots &f^{(k)}_{s+k/2+m-1}\\
f^{(2)}_{s+k/2+m-2}&\dots &f^{(k)}_{s+k/2+m-2}\\
\vdots& &\vdots\\
\fh^{(2)}_{s-1}&\dots &\fh^{(k)}_{s-1}\\
\vdots& &\vdots\\
f^{(2)}_{s-k/2+m}&\dots &f^{(k)}_{s-k/2+m}
\end{matrix}\right|f^{(k+1)}_{s-1}\Bigg]\dms,
\end{split}
\end{equation}
for $m=0,1,\dots,k/2$. For simplicity, we wrote explicitly only the 
terms with $f^{(1)}_n$ in front of the integral, but we have also similar 
expressions obtained by exchanging the roles of $f^{(1)}_n$ and $f^{(j)}_n$ 
for every $j=2,\dots,k$.

\item[(iv)] Finally if $k$ is odd, we shall use (ii) with $l=(k-1)/2$ and 
$l=(k-3)/2$ and write $F_n$ as the average of these two sums. Then we can 
apply the same procedure that we used in (iii) to write $F_n^{(k)}$ as sums of 
integrals, whose integrands can be combined in pairs involving 
$f^{(k+1)}_{s}$ and $f^{(k+1)}_{s-1}$.
\end{Remark}
The second important ingredient needed for the proof of \thref{th4.2} is the 
following lemma.

\begin{Lemma}\label{le4.5}
For $r(n)\in\Rset[n]$  there exists $\cBb\in\fDb$ such that
\begin{equation}\label{4.16}
\int_{-1}^{n}[r(s)\pab_{s}(z)
-r(-s-\al-\be)\pab_{s-1}(z)]\dms=\cBb\pab_n(z).
\end{equation}
\end{Lemma}
\begin{proof} 
Let us rewrite $r(n)$ as a polynomial of $2n+\al+\be$, i.e. we set 
$$r(n)=\rb(2n+\al+\be).$$
Then $r(-n-\al-\be)=\rb(-(2n+\al+\be))$. Thus, it is enough to show 
that for every $j\in\Nset_0$ there is $\cBb_j\in\fDb$ such that 
\begin{equation}\label{4.17}
(2n+\al+\be)^j\left(\pab_{n}(z)+(-1)^{j+1}\pab_{n-1}(z)\right)
=\cBb_j\left(\pab_{n}(z)-\pab_{n-1}(z)\right).
\end{equation}
Indeed, if \eqref{4.17} holds then 
\begin{equation*}
\begin{split}
&\int_{-1}^{n}(2s+\al+\be)^j\left(\pab_{s}(z)+(-1)^{j+1}\pab_{s-1}(z)\right)
\dms\\
&\qquad\qquad=\cBb_j\int_{-1}^{n}(\pab_{s}(z)-\pab_{s-1}(z))\dms\\
&\qquad\qquad=\cBb_j\pab_{n}(z),
\end{split}
\end{equation*}
giving the proof when $\rb(t)=t^j$, hence for arbitrary polynomials by 
linearity. When $j$ is odd, \eqref{4.17} follows immediately from \eqref{2.9}
and the definition of $\fDb$, see equations \eqref{4.1}-\eqref{4.2}. 
The case $j=0$ is obvious and therefore it remains to prove the statement 
when $j>0$ is even. Note that $(2n+\al+\be)=\la_{n}-\la_{n-1}$. Thus 
\eqref{4.17} for $j$ even will follow if we can show that there exist 
operators $\cBb',\cBb''\in\fDb$ such that
\begin{subequations}\label{4.18}
\begin{align}
(\la_{n}+\la_{n-1})\left(\pab_{n}(z)-\pab_{n-1}(z)\right)
&=\cBb'\left(\pab_{n}(z)-\pab_{n-1}(z)\right)\label{4.18a}\\
\intertext{and}
\la_{n}\la_{n-1}\left(\pab_{n}(z)-\pab_{n-1}(z)\right)
&=\cBb''\left(\pab_{n}(z)-\pab_{n-1}(z)\right).\label{4.18b}
\end{align}
\end{subequations}
Using \eqref{4.1} we see that the Jacobi operator $\Bab$ defined in 
\eqref{2.2} belongs to $\fDb$ since 
\begin{equation}\label{4.19}
\Bab=\cD_{1}^2+(\al+\be+1)\cD_{1}-\cD_{2}.
\end{equation}
From equations \eqref{2.4} and \eqref{2.9} we find
\begin{equation*}
\begin{split}
\la_{n-1}\pab_{n}(z)-\la_{n}\pab_{n-1}(z)=&
(\la_{n-1}-\la_{n})\left(\pab_{n}(z)+\pab_{n-1}(z)\right)\\
&\qquad\qquad+\la_{n}\pab_{n}(z)-\la_{n-1}\pab_{n-1}(z)\\
=&\cBb'''\left(\pab_{n}(z)-\pab_{n-1}(z)\right),
\end{split}
\end{equation*}
where $\cBb'''=\Bab-2\cD_{1}-(\al+\be)\in\fDb$. Using the above equation 
together with \eqref{2.4} we obtain
\begin{equation*}
\begin{split}
(\la_{n}+\la_{n-1})\left(\pab_{n}(z)-\pab_{n-1}(z)\right)=&
\la_{n}\pab_{n}(z)-\la_{n-1}\pab_{n-1}(z)\\
&\qquad\qquad+ \la_{n-1}\pab_{n}(z)-\la_{n}\pab_{n-1}(z)\\
=&\cBb'\left(\pab_{n}(z)-\pab_{n-1}(z)\right),
\end{split}
\end{equation*}
where $\cBb'=\Bab+\cBb'''\in\fDb$ proving \eqref{4.18a}. Similarly, 
\begin{equation*}
\begin{split}
\la_{n}\la_{n-1}\left(\pab_{n}(z)-\pab_{n-1}(z)\right)=&
\Bab\left(\la_{n-1}\pab_{n}(z)-\la_{n}\pab_{n-1}(z)\right)\\
=&\cBb''\left(\pab_{n}(z)-\pab_{n-1}(z)\right),
\end{split}
\end{equation*}
where $\cBb''=\Bab\cBb'''\in\fDb$, establishing \eqref{4.18b} and completing 
the proof.
\end{proof}

\subsection{Proof of the main one-dimensional theorem}\label{ss4.3}
We are now ready to give the proof of \thref{th4.2}. Let $f\in\cAaba$. From \reref{re4.1} we know that, up to an additive constant, we have 
\begin{equation*}
f(\la_{n-k/2})=\int_{-1}^{n-1}\ep^{(k+2)}_{s+1}g(\la_{s-(k-1)/2})\tau_{s}\dms.
\end{equation*}
We apply \leref{le4.3} and \reref{4.4} (i) with 
\begin{align*}
f^{(0)}_n&=\ep^{(k+2)}_{n+1}g(\la_{n-(k-1)/2})\\
f^{(j)}_n&=\psi^{(j-1)}_n\text{ for }j=1,2,\dots,k\\
f^{(k+1)}_n&=\pab_n(z).
\end{align*}
Using equations \eqref{3.10} and \eqref{4.4} we see that the right-hand side 
of \eqref{4.10} is equal to $(f(\la_{n-k/2})+c)q_n(z)$, where $c$ is a constant 
(independent of $n$ and $z$). The goal now is to show that if we 
choose appropriately the integers $n_j$ in \leref{le4.3}, then there exists 
a differential operator $\cB_f\in\fDb$ such that 
\begin{equation}\label{4.20}
F_n=\cB_f \,\pab_n(z).
\end{equation}
Suppose first that $k$ is even and let us write $F_n$ as explained in 
\reref{re4.4} (iii). From \reref{re4.1} it follows that the integral in the last 
term in the sum \eqref{4.14} representing $F_n$ is an element of $\Rset[\la_n]$.  
Therefore, using \eqref{2.4} and \eqref{4.19} we see that the last term in this sum is of the 
form $\cB \pab_n(z)$ for some operator $\cB\in\fDb$.

Thus we can consider the sum $F^{(k)}_{n}$ of the first $k$ terms. 
It is enough to show that each term of the form \eqref{4.15} can be 
represented as $\cB \pab_n(z)$ for some operator $\cB\in\fDb$. Recall that 
$f^{(j)}_n\in\Rset[\la_n]$ for $j=1,2,\dots, k$ and therefore 
it suffices to show that the integral is of the form $\cB \pab_n(z)$ for 
$\cB\in\fDb$ (since then we can commute $\cB$ and $f^{(j)}_n$ and use 
\eqref{2.4}). Now we apply \leref{le4.5}. To simplify the argument, we shall 
use the involution $I^{(\al+\be-1)}$ which acts on polynomials in $\Rset[n]$ by 
$I^{(\al+\be-1)}(n)=-(n+\al+\be)$. Thus for $r(n)\in\Rset[n]$ we have 
$I^{(\al+\be-1)}(r(n))=r(-(n+\al+\be))$. Note that for every $l\in\Rset$ we 
have 
$$I^{(\al+\be-1)}(\la_{n+l})=\la_{n-l-1}.$$
Therefore, if we denote by $\det{}'_n$ and $\det{}''_n$ the determinants in 
\eqref{4.15}, then $I^{(\al+\be-1)}$ will reverse the order of the rows, hence 
$$I^{(\al+\be-1)}(\det{}'_n)=(-1)^{(k-1)(k-2)/2}\det{}''_n.$$
Since 
$$f^{(0)}_{n\pm m+k/2-1}=\ep^{(k+2)}_{n\pm m+{k/2}}g(\la_{n\pm m-1/2}),$$
and 
$$I^{(\al+\be-1)}(g(\la_{n+ m-1/2}))= g(\la_{n- m-1/2}),$$
it remains to check that 
$$I^{(\al+\be-1)}(\ep^{(k+2)}_{n+m+{k/2}})
= (-1)^{(k-1)(k-2)/2}\ep^{(k+2)}_{n-m+{k/2}},$$
which follows at once from \eqref{4.5} by considering the two possible 
cases $k\equiv 0\mod 4$ and $k\equiv 2\mod 4$.

The case when $k$ is odd can be handled in a similar manner, using 
\reref{re4.4} (iv). \qed

\begin{Remark}\label{re4.6}
If $\beta\in \Nset$ then instead of the Darboux transformations \eqref{3.1} at $z=1$ we can consider a sequence of Darboux transformations at $z=0$. \thref{th4.2} can directly be applied in this case by replacing $z$ with $1-z$ and exchanging the roles of $\alpha$ and $\beta$. In particular, the corresponding commutative algebra $\cDaba$ constructed in \thref{th4.2} will be a subalgebra of $\fDah=\Rset\langle \cDh_1,\cDh_2\rangle$, where 
$$\cDh_1=(z-1)\pd_z \text{ and }\cDh_2=(1-z)\pd_z^2-(\alpha+1)\pd_z.$$
Finally, if both $\alpha$ and $\beta$ are positive integers then we can iterate the Darboux transformation at $z=1$ and $z=0$ at most $\alpha$ and $\beta$ times, respectively. In this case, we need also the functions
\begin{subequations}\label{4.21}
\begin{align}
\varphi^{1,j}_n
&=\frac{(-1)^n\be!(n+1)_{\al}(n+1)_j(-n-\al-\be)_{j}}{j!\al!(1-\be)_j(n+1)_{\be}},\label{4.21a}\\
\varphi^{2,j}_n
&=\frac{(-1)^n(n+1)_{\al+\be}(-n)_j(n+\al+\be+1)_j}
{j!(\al+\be)!(1+\be)_{j}},\label{4.21b}
\end{align}
\end{subequations}
which are linearly independent  and satisfy
\begin{align*}
&\cL_{\al,\be}(n+\var,E_n)\varphi^{i,0}_{n+\var}=0,\text{ for }i=1,2\\
&\cL_{\al,\be}(n+\var,E_n)\varphi^{i,j}_{n+\var}=\varphi^{i,j-1}_{n+\var},
\text{ for }i=1,2 \text{ and }1 \leq j<\be.
\end{align*}
Note that if we define 
$$\xi_n=(-1)^{n}\frac{(n+1)_{\be}}{(n+1)_{\al}},$$
then $\xi_n\varphi^{i,j}_n$ become polynomials of $\la_{n}$. If we apply $k$ Darboux steps at $1$ and 
$l$ Darboux steps at $0$, then instead of \eqref{4.4} we define $\tau_n$ by
\begin{equation*}
\tau_n=\left(\xi_n(n+\al-k-l+2)_{k+l-1}\right)^l\Wr_n(\psi^{(0)}_n,\psi^{(1)}_n,\dots,\psi^{(k-1)}_n,\hat{\psi}^{(0)}_n,\hat{\psi}^{(1)}_n,\dots,\hat{\psi}^{(l-1)}_n),
\end{equation*}
where $\hat{\psi}^{(j)}_n$ are defined similarly to $\psi^{(j)}_n$, using the functions $\varphi^{i,j}_n$. Then \thref{th4.2} holds with $\fDb$ replaced by $\fD=\Rset\langle z\pd_z,\pd_z\rangle$, since we need to use now operators from both algebras $\fDb$ and $\fDah$. The proof follows along the same lines, using \leref{le4.5} together with the analogous statement concerning the transformations at $0$, i.e. for $r(n)\in\Rset[n]$  there exists $\cBb\in\fDah$ such that
\begin{equation*}
\frac{1}{\xi_n}\int_{-1}^{n}[r(s)\xi_{s}\pab_{s}(z)
-r(-s-\al-\be)\xi_{s-1}\pab_{s-1}(z)]\dms=\cBb\pab_n(z).
\end{equation*}
\end{Remark}

Note that the construction of the operator $B_f(\cD_1,\cD_2)$ in the proof of \thref{th4.2} depends only on equations \eqref{2.4}, \eqref{2.9} and \eqref{4.19}. For $s\in\Nset_0$ let us denote
\begin{equation}\label{4.22}
\cD_{2,s}=z\pd_z^2+(\be+1)\pd_z-\frac{s(s+2\beta)}{4z}=\cD_2-\frac{s(s+2\beta)}{4z},
\end{equation}
and let us define $\Babs$ similarly to $\Bab$ using \eqref{4.19} with $\cD_2$ replaced by $\cD_{2,s}$:
\begin{equation}\label{4.23}
\Babs=\cD_1^2+(\al+\be+1)\cD_1-\cD_{2,s}.
\end{equation}
Then it is easy to check that 
$$z^{-s/2}\Babs\cdot z^{s/2}=B_{\al,\be+s}(z,\pd_z)+\la^{\al+\be}_{n+s/2}-\la^{\al+\be+s}_{n}.$$
From this relation and \eqref{2.4} it follows immediately that
\begin{equation}\label{4.24}
\Babs \left[p^{\al,\be+s}_n(z)z^{s/2}\right]=\la^{\al+\be}_{n+s/2}p^{\al,\be+s}_n(z)z^{s/2}.
\end{equation}
It is also easy to see that 
\begin{equation}\label{4.25}
\begin{split}
&(2z\pd_z+\al+\be)(p^{\al,\be+s}_{n}(z)z^{s/2}+p^{\al,\be+s}_{n-1}(z)z^{s/2})\\
&\qquad =(2(n+s/2)+\al+\be)(p^{\al,\be+s}_{n}(z)z^{s/2}-p^{\al,\be+s}_{n-1}(z)z^{s/2}).
\end{split}
\end{equation}
Note also that 
$$[\cD_{2,s},\cD_1]=\cD_{2,s},$$
which shows that for every $f\in\cAaba$ we have a well-defined operator $B_f(\cD_1,\cD_{2,s})$.
Comparing equations  \eqref{2.4}, \eqref{2.9} and \eqref{4.19} with \eqref{4.24}, \eqref{4.25} and \eqref{4.23} 
we obtain the following corollary of \thref{th4.2}.
\begin{Proposition}\label{pr4.7}
If we define for $s\in\Nset_0$
\begin{equation}\label{4.26}
\qh^{\al,\be;a}_{n,s}(z)=\Wr_n(\psi^{(0);\al,\be;a}_{n+s/2},\psi^{(1);\al,\be;a}_{n+s/2},\dots,\psi^{(k-1);\al,\be;a}_{n+s/2},p^{\al,\be+s}_n(z)),
\end{equation}
then for every $f\in\cAaba$ we have
\begin{equation}\label{4.27}
B_f(\cD_1,\cD_{2,s})\left[\qh^{\al,\be;a}_{n,s}(z)z^{s/2}\right]=f(\la^{\al+\be}_{n+(s-k)/2})\qh^{\al,\be;a}_{n,s}(z)z^{s/2},
\end{equation}
where $B_f$ are the operators constructed in \thref{th4.2}.
\end{Proposition}
We have displayed all parameters in \eqref{4.26} to underline the fact that in the functions 
$\psi^{(j)}_n$ the parameters stay the same, only the variable $n$ is shifted by $s/2$, while in $\pab_n(z)$ we change only the parameter $\be$.

\section{Krall-Jacobi algebras in higher dimension}\label{se5}
In this section we show that the commutative algebra $\cDaba$ of ordinary differential operators constructed in \thref{th4.2} is isomorphic to a commutative algebra of partial differential operators invariant under rotations which can be diagonalized in the space of polynomials in $d$ variables.

\subsection{Notations}\label{ss5.1} 
Let $x=(x_1,x_2,\dots,x_d)\in\Rset^d$, and let $\cPd=\Rset[x]=\Rset[x_1,x_2,\dots,x_d]$ be the corresponding ring of polynomials in the variables $x_1,x_2,\dots,x_d$.  We denote by $B^d$ and $S^{d-1}$ the unit ball and the unit sphere in $\Rset^{d}$:
$$B^{d}=\{x\in\Rset^d:||x||\leq 1\}, \qquad S^{d-1}=\{x\in\Rset^d:||x||= 1\}.$$
In polar coordinates we shall write $x=\rho x'$ where $\rho=||x||$ and $x'\in S^{d-1}$.

We denote by $\De_x=\sum_{j=1}^{d}\pd_{x_{j}}^2$ the Laplace operator and by $\cH_l$ the space of homogeneous harmonic polynomials of degree $l$, i.e. the homogeneous polynomials $Y(x)$ of degree $l$, satisfying the equation $\De_x Y(x)=0$. It is well known that the dimension  $\sigma_l=\dim \cH_l$ is given by 
$$\sigma_l=\binom{l+d-1}{d-1}-\binom{l+d-3}{d-1}.$$
The restrictions of $Y\in H_l$ on $S^{d-1}$ are the spherical harmonics. Let $\omega$ denote the Lebesgue measure on $S^{d-1}$ and let $\omega_d:=\omega(S^{d-1})=2\pi^{d/2}/\Gamma(d/2)$.
Throughout this section, we use $\{Y^l_j(x):1\leq j\leq \sigma_l\}$ to denote an orthonormal basis for $\cH_l$ on $S^{d-1}$. 
Thus, we have
\begin{equation}\label{5.1}
\frac{1}{\omega_d}\int_{S^{d-1}}Y^{l_1}_{j_1}(x')Y^{l_2}_{j_2}(x')d\omega(x')=\delta_{l_1,l_2}\delta_{j_1,j_2}.
\end{equation}
Recall that in polar coordinates we have
\begin{equation}\label{5.2}
\De_x=\pd_{\rho}^2+\frac{d-1}{\rho}\pd_{\rho}+\frac{1}{\rho^2}\De_{S^{d-1}},
\end{equation}
where $\De_{S^{d-1}}$ is the Laplace-Beltrami operator on the sphere $S^{d-1}$. Since 
\begin{equation}\label{5.3}
Y^{l}_{j}(x)=\rho^{l}Y^{l}_{j}(x'),
\end{equation}
the polynomials $Y^{l}_j(x)$ satisfy the equation
\begin{equation}\label{5.4}
\De_{S^{d-1}}Y^{l}_{j}(x)=-l(l+d-2)Y^{l}_{j}(x).
\end{equation}
\subsection{Construction of the algebra of partial differential operators}\label{ss5.2} 
To motivate the construction, notice that if we set $\be=\frac{d}{2}-1$ and if we replace $z\in[0,1]$ by $\rho\in[0,1]$, where $z=\rho^2$ then for the Jacobi measure on $[0,1]$ considered in \seref{se2} we obtain
\begin{equation}\label{5.5}
(1-z)^{\alpha}z^{\beta}dz=2(1-\rho^2)^{\alpha}\rho^{d-1}d\rho.
\end{equation}
Recall that the Jacobi polynomials on $B^{d}$ are defined as orthogonal polynomials on $B^{d}$ with respect to measure $d \mu_{\al}(x)=(1-||x||^2)^{\alpha}dx$, see for instance \cite[page~38]{DX}. If we use polar coordinates $x=\rho x'$ then, up to a scaling factor, $d \mu_{\al}(x)$ is a product of the measure in \eqref{5.5} on $[0,1]$ and the surface measure $d\omega (x')$ on $S^{d-1}$. Note also that the operators $\cD_1$ and $\cD_2$ defined in the beginning of \seref{se4} change as follows:
\begin{equation}\label{5.6}
\begin{split}
&\cD_{1}=z\pd_z=\frac{1}{2}\rho\pd_{\rho}\\
&\cD_{2}=z\pd_z^2+(\be+1)\pd_z= \frac{1}{4}\left[\pd_{\rho}^2+\frac{d-1}{\rho}\pd_{\rho}\right].
\end{split}
\end{equation}
Moreover, in polar coordinates we have 
$$\rho\pd_{\rho}=\sum_{j=1}^{d}x_j\pd_{x_j}=x \cdot \na_x,$$
where $\na_x$ is the gradient, while the operator $\pd_{\rho}^2+\frac{d-1}{\rho}\pd_{\rho}$ is the radial part of the Laplace operator $\De_x$.
It is easy to see that the operators $ \frac{1}{2}x \cdot \na_x$ and $\frac{1}{4} \De_x$ satisfy the commutativity relation
\begin{equation*}
\left[\frac{1}{4} \De_x, \frac{1}{2}x \cdot \na_x \right]=\frac{1}{4} \De_x,
\end{equation*}
which combined with \eqref{4.3} shows that there is a natural isomorphism between the algebra $\fDb$ given in \eqref{4.2} and the associative algebra generated by 
$\frac{1}{2}x \cdot \na_x$ and $\frac{1}{4} \De_x$ defined by
\begin{subequations}\label{5.7}
\begin{align}
&\cD_{1}=z\pd_z\rightarrow \frac{1}{2}x \cdot \na_x \label{5.7a}\\
&\cD_{2}=z\pd_z^2+(\be+1)\pd_z\rightarrow \frac{1}{4} \De_x.\label{5.7b}
\end{align}
\end{subequations}
Thus if we set 
\begin{equation*}
\cAd(\al;a)=\mathcal{A}^{\al,d/2-1;a}
\end{equation*}
we see that the commutative algebra $\mathcal{D}^{\al,d/2-1;a}$ defined in \thref{th4.2} is isomorphic to a commutative algebra of partial differential operators:
\begin{equation}\label{5.8}
\cKd(\al;a)=\left\{B_{f}\left(\frac{1}{2}x \cdot \na_x\, ,\,\frac{1}{4}\De_x\right):f\in\cAd(\al;a)\right\}.
\end{equation}
The main point now is that we can write a basis for $\cPd$ which diagonalizes the operators in $\cKd(\al;a)$, using the polynomials $\qh$ defined in \prref{pr4.7} and the spherical harmonics.

\begin{Theorem}\label{th5.1}
For $n\in\Nset_0$, $i\in\Nset_0$, such that $i\leq \frac{n}{2}$ and $j\in\{1,2,\dots,\sigma_{n-2i}\}$ define
\begin{equation}\label{5.9}
Q_{n,i,j}(x)=\qh^{\al,d/2-1;a}_{i,n-2i}(||x||^2)\,Y^{n-2i}_{j}(x).
\end{equation}
Then the polynomials $\{Q_{n,i,j}(x)\}$ form a basis for $\cPd$ and for every $f\in\cAd(\al;a)$ we have
\begin{equation}\label{5.10}
B_{f}\left(\frac{1}{2}x \cdot \na_x\,,\,\frac{1}{4}\De_x\right)Q_{n,i,j}(x)=f(\la^{\al+d/2-1}_{(n-k)/2})Q_{n,i,j}(x).
\end{equation}
\end{Theorem}
\begin{proof}
The linear independence of $Q_{n,i,j}$ follows easily from \eqref{5.1} and the fact that the polynomials $\{\qh^{\al,\be;a}_{i,s}\}_{i\in\Nset_0}$ are linearly independent. Since 
\begin{align*}
\sum_{i=0}^{\lfloor n/2\rfloor}\sigma_{n-2i}=\binom{n+d-1}{d-1} =\text{the number of monomials of total degree }n,
\end{align*}
we see that the polynomials $Q_{n,i,j}$ form a basis for $\cPd$. It remains to show that they satisfy \eqref{5.10}. 
Using polar coordinates we find
\begin{equation*}
Q_{n,i,j}(x)=\qh^{\al,d/2-1;a}_{i,n-2i}(\rho^2)\rho^{n-2i}\,Y^{n-2i}_{j}(x').
\end{equation*}
From equations \eqref{5.2} and \eqref{5.4} we see that 
\begin{equation*}
\begin{split}
&\De_{x}Q_{n,i,j}(x)\\
&\quad=\left[\pd_{\rho}^2+\frac{d-1}{\rho}\pd_{\rho}-\frac{1}{\rho^2}(n-2i)(n-2i+d-2)\right]\qh^{\al,d/2-1;a}_{i,n-2i}(\rho^2)\rho^{n-2i}\,Y^{n-2i}_{j}(x') .
\end{split}
\end{equation*}
The proof of \eqref{5.10}  now follows from \prref{pr4.7} upon changing $z=\rho^2$ and using equations \eqref{5.6}.
\end{proof}

\section{An explicit example}\label{se6}

In this section we illustrate all steps in the paper with the simplest nontrivial case $\al=k=1$. 

\subsection{Krall polynomials} \label{ss6.1}
Let us first consider the one-dimensional case which leads to Krall polynomials \cite{Kr2}.  When $\al=k=1$ the polynomials $q_n(z)$ constructed in \seref{se3} are given by the following formula

\begin{equation}\label{6.1}
q^{1,\be;a_0}_n(z)=\left| \begin{matrix}\psi^{(0);1,\be;a_0}_{n} &p^{1,\be}_{n}(z)\\ \psi^{(0);1,\be;a_0}_{n-1} &p^{1,\be}_{n-1}(z)\end{matrix}\right|,
\end{equation}
where $p^{1,\be}_n(z)$ are the Jacobi polynomials in \eqref{2.1} and
\begin{equation}\label{6.2}
\psi^{(0);1,\be;a_0}_{n}=a_0+\frac{(n+1)(n+\be+1)}{\be+1}.
\end{equation}
For $n\neq m$ they satisfy the orthogonality relation
\begin{equation}\label{6.3}
\int_{0}^{1}q^{1,\be;a_0}_n(z)q^{1,\be;a_0}_m(z)z^{\be}dz+\frac{1}{a_0(\be+1)}q^{1,\be;a_0}_n(1)q^{1,\be;a_0}_m(1)=0.
\end{equation}
Since $k=1$ formula \eqref{4.4} shows that $\tau_n=\psi^{(0);1,\be;a_0}_{n}=a_0+\frac{(n+1)(n+\be+1)}{\be+1}$.
From this and \reref{re4.1} it follows that the commutative algebra $\cA^{1,\be;a_0}$ defined in \eqref{4.7} is generated by two polynomials of degrees $2$ and $3$. A short computation shows that 
\begin{equation}\label{6.4}
\cA^{1,\be;a_0}=\Rset[f_2,f_3], 
\end{equation}
where 
\begin{subequations}\label{6.5}
\begin{align}
f_2(t)&= t^2 + \frac{1}{2}(3 + 4a_0 + 4\be + 4a_0\be)t\label{6.5a}\\
f_3(t)&=t^3 + \frac{1}{4}(1 + 6a_0 + 6\be + 6a_0\be)t^2 - \frac{1}{16}(21 + 12a_0 + 28\be + 12a_0\be + 4\be^2)t.\label{6.5b}
\end{align}
\end{subequations}
The algebra $\cD^{1,\be;a_0}$ defined in \thref{th4.2} is generated by the operators $B_2:=B_{f_2}$ and $B_3:=B_{f_3}$ of orders 4 and 6 respectively. The operator $B_2$ goes back to the work of Krall \cite{Kr2}. Using the operators $\cD_1$ and $\cD_2$ given in \eqref{4.1} we can write $B_2$ as follows:
\begin{equation}\label{6.6}
\begin{split}
&B_2(\cD_1,\cD_2)=\cD_1^{4}-2\cD_2\cD_1^{2}+\cD_2^{2}+2(1+\be)\cD_1^{3}-2\be\cD_2\cD_1\\
&\qquad 
+(1 + 2a_0 + 3\be + 2a_0 \be + \be^2)\cD_1^2 - 
 2(1 + a_0 + a_0\be)\cD_2 \\
 &\qquad+ (1 + \be) (\be + 2 a_0 (1 + \be))\cD_1 - \frac{1}{16} (3 + 2 \be) (3 + 6 \be + 8 a_0 (1 + \be)).
\end{split}
\end{equation}
One can write a similar formula for $B_3$. If we denote 
$$a_0^{s}=\frac{4a_0+4a_0\be+2\be s+s^2}{4(1+\be+s)},$$
then from formula \eqref{6.2} it follows easily that 
\begin{equation*}
\psi^{(0);1,\be;a_0}_{n+s/2}=\frac{1+\be+s}{1+\be}\psi^{(0);1,\be+s;a_0^{s}}_{n},
\end{equation*}
i.e. the functions $\psi^{(0);1,\be;a_0}_{n+s/2}$ and $\psi^{(0);1,\be+s;a_0^{s}}_{n}$ differ by a factor independent of $n$. Therefore for the polynomials $\qh$ defined in \eqref{4.26} we find
\begin{equation*}
\qh^{1,\be;a_0}_{n,s}(z)=\frac{1+\be+s}{1+\be}q^{1,\be+s;a_0^{s}}_{n}(z).
\end{equation*}
This combined with \eqref{6.3} shows that for $n\neq m$ the polynomials $\qh^{1,\be;a_0}_{n,s}(z)$ satisfy the orthogonality relation
\begin{equation}\label{6.7}
\left(1+\frac{s^2+2\be s}{4a_0(\be+1)}\right)\int_{0}^{1}\qh^{1,\be;a_0}_{n,s}(z)\qh^{1,\be;a_0}_{m,s}(z)z^{\be+s}dz+\frac{1}{a_0(\be+1)}\qh^{1,\be;a_0}_{n,s}(1)\qh^{1,\be;a_0}_{m,s}(1)=0.
\end{equation}

\subsection{Krall polynomials in higher dimension} \label{ss6.2}
Let us consider now $x\in\Rset^d$ and set $\be=\frac{d}{2}-1$. Then the algebra $\cKd(\al,a_0)$ defined in \eqref{5.8} is 
generated by the operators $B_{2}\left(\frac{1}{2}x \cdot \na_x\, ,\,\frac{1}{4}\De_x\right)$ and $B_{3}\left(\frac{1}{2}x \cdot \na_x\, ,\,\frac{1}{4}\De_x\right)$ which act diagonally on the basis of polynomials $Q_{n,i,j}$ described in \thref{th5.1}. 
Let us denote by $u_0$ the constant
$$u_0=\frac{1}{a_0(\be+1)}=\frac{2}{a_0d}.$$
Then using equations \eqref{5.1}, \eqref{5.4}, \eqref{5.5} and \eqref{6.7} we see that the polynomials $Q_{n,i,j}$ are mutually orthogonal with respect to the inner product on $\cPd$ defined by
\begin{equation}\label{6.8}
\begin{split}
\langle f, g\rangle = &\int_{B^{d}}f(x)g(x)dx+\frac{u_0}{2} \int_{S^{d-1}}f(x')g(x')d\omega(x')\\
&\qquad -\frac{u_0}{4} \int_{B^{d}}\left(\De_{S^{d-1}} f(x)\right) g(x)dx.
\end{split}
\end{equation}
The interesting new phenomenon in the multivariate case is the fact that, even in the simplest example ($\al=k=1$), the polynomials are orthogonal with respect to an inner product involving the spherical Laplacian. In this respect, the multivariate analogs of Krall polynomials discussed here are related to the so called Sobolev orthogonal polynomials, see for instance \cite{X} and the references therein. However the appearance of $\De_{S^{d-1}}$ in the inner product defined in \eqref{6.8}, which comes naturally from our approach, seems to be new. 

It would be interesting to find explicit orthogonality relations for the general multivariate polynomials $Q_{n,i,j}$ defined in \thref{th5.1}. The key step would be to discover an orthogonality relation similar to \eqref{6.7} for the polynomials $\qh^{\al,\be;a}_{n,s}$ given in \prref{pr4.7}.

\end{document}